\documentclass{article}
\usepackage{amsmath}
\usepackage{amssymb, mathrsfs}
\usepackage{amsthm}
\usepackage{mathtools}
\usepackage{tikz}
\usepackage{tikz-cd}
\usepackage[
backend=biber,
style=alphabetic,
]{biblatex}

\newtheorem{theorem}{Theorem}[section]
\newtheorem{lemma}[theorem]{Lemma}
\newtheorem{proposition}[theorem]{Proposition}
\newtheorem*{theorem*}{Theorem}

\theoremstyle{definition}
\newtheorem{definition}[theorem]{Definition}

\newtheorem{notation}[theorem]{Notation}
\newtheorem{remark}[theorem]{Remark}

\DeclarePairedDelimiter\norm{\lVert}{\rVert}

\title{Admissibility over Semi-Global Fields in the Bad Characteristic Case}
 \author{Yael Davidov}

\begin{document}
\maketitle

\begin{abstract}
A finite group $G$ is said to be admissible over a field $F$ if there exists a division algebra $D$ central over $F$ with a maximal subfield $L$ such that $L/F$ is Galois with group $G$. In this paper we give a complete characterization of admissible groups over function fields of curves over equicharacteristic complete discretely valued fields with algebraically closed residue fields, such as the field $\overline{\mathbb{F}_P}((t))(x)$.\\
\end{abstract}

\section{Introduction}
Admissibility was first studied over global fields by Schacher \parencite{SCHACHER1968451}. He found that the $q$-Sylow subgroups of an admissible group over a characteristic $p$ global field had to be metacyclic (i.e. the extension of a cyclic group by another cyclic group) for all primes $q\neq p$. Sufficient conditions for admissibility over global fields is still an open question, with results for solvable groups given by Sonn \parencite{sonn1983q}. Admissibility over other types of fields have been studied by Harbater, Hartmann and Krashen \parencite{harbater2011patching}, Suresh and Reddy \parencite{reddy2013admissibility} and Neftin and Paran \parencite{neftin2010patching}.\\

In 2011 Harbater, Hartmann, and Krashen used field patching techniques to study admissibility over semi-global fields \parencite{harbater2011patching}. A semi-global field $F$ is the function field of a curve over a complete discretely valued field $K$. Harbater, Hartmann, and Krashen showed that if $G$ is admissible over a semi-global field $F$, and the residue field of $K$, $k$, is algebraically closed, then the $q$-Sylow subgroups of $G$ must be abelian of rank at most 2 for every $q\neq p:=\text{char}(k)$. As in the global field case, it was not clear what restrictions, if any, should be placed on the $p$-Sylow subgroups.\\

The main contribution of this work is to show that, in the equicharacteristic $p$ case, the structure of the $p$-Sylow subgroup of a group $G$ does not impact the admissibility of $G$ over $F$. With this result in hand we can completely characterize admissible groups over a new class of semi-global fields.

\begin{theorem*}[Theorem \ref{mainresult}]
Let $K$ be a complete discretely valued field with algebraically closed residue field $k$, and let F be a finitely generated field extension of K with transcendence degree 1. Assume further that $\text{char}(K)=\text{char}(k)$. A finite group $G$ is admissible over $F$ if and only if the $q$-Sylow subgroups of $G$ are abelian of rank at most 2 for every prime $q\neq \text{char}(k)$. 
\end{theorem*}

This result builds on the work of Guico \parencite{erneststhesis} and uses field patching techniques as used in \parencite{harbater2011patching}. One difficulty in addressing this case is that while Guico's work establishes admissibility for $p$-groups, it leaves open the question of how to patch these compatibly to obtain more general groups. An important part of our strategy is changing the patching model by blowing up and constructing compatible objects on intermediate patches using generic Galois theory.\\

In section 2 we provide a quick review of the theory of field patching which will be the primary technique used in the proof of the main result and establish some notation. Section 3 establishes the local admissibility of the $p$-Sylow subgroup, recalling results of Guico. In Section 4 we prove some technical results which establish control over the behavior of the $p$-Sylow piece on overlapping patches. In Section 5 we construct compatible structures over a sequence of patches using generic Galois theory. The proof of the main result is given in section 6.

\section{Field Patching}\label{patching}
The primary underlying framework in the proof of the main result is the field patching technique introduced by Harbater and Hartmann in 2007. We will be using a streamlined setup described in \parencite{harbater2011patching} and \parencite{harbater2015local}. Broadly speaking, field patching allows one to prove the existence of certain algebraic objects over a given field by constructing corresponding objects over a special system of field extensions in a compatible way, when such a system of overfields exists.

\begin{definition} 
Let $I$ be a finite partially ordered set with a field $F_i$ for each $i\in I$. Suppose for each $i\succ j$ there is an inclusion $\iota_{ij}:F_i\rightarrow F_j$. We call this a system of fields and denote it by $\mathscr{F}=\{F_i\}_{i\in I}$. Suppose also that $\mathcal{A}(F_i)$ is a category which we think of as consisting of ``algebraic objects of type $\mathcal{A}$ over $F_i$" such that extension of scalars $\mathcal{E}_{ij}:\mathcal{A}(F_i)\rightarrow \mathcal{A}(F_j)$ given by $A\mapsto A\otimes_{F_i}F_j$ is a functor for all $i\succ j$.
\end{definition}

\begin{remark}
An example of the type of algebraic objects that might be considered are finite dimensional vector spaces. In this case $\mathcal{A}(F_i)$ would denote the category of finite dimensional vector spaces over $F_i$. More generally, we can consider vector spaces with a fixed tensor as described by Serre in Section X.2 of \parencite{serre2013local}. For a discussion of the types of algebraic objects that one can consider when using field patching techniques see Definition 3.2.1 in \parencite{krashen2010field}.
\end{remark}

\begin{definition}
An $\mathcal{A}$-patching problem for the system of fields $\mathscr{F}=\{F_i\}_{i\in I}$ is a set of objects $\{V_i\}_{i\in I}$ such that $V_i\in\mathcal{A}(F_i)$ with isomorphisms $\phi_{ij}:V_i\otimes_{F_i}F_j\rightarrow V_j$ for all $i\succ j$. We define a morphism of $\mathcal{A}$-patching problems to be a collection of maps $\{\alpha_i\}_{i\in I}$ where $\alpha_i\in \text{Hom}_{\mathcal{A}(F_j)}(V_i,V'_i)$ such that for all $i\succ j$ in $I$ the following diagram commutes:
\[
	\begin{tikzcd}
V_i\otimes_{F_i}F_j \arrow{r}{\phi_{ij}} \arrow[swap]{d}{\alpha_i\otimes id} & V_j \arrow{d}{\alpha_j} \\
V'_i\otimes_{F_i}F_j \arrow{r}{\phi'_{ij}} & V'_j.
	\end{tikzcd}
\] 
This forms a category of $\mathcal{A}$-patching problems over $\mathscr{F}$ which we denote $PP_{\mathcal{A}}(\mathscr{F})$. If $\mathscr{F}$ is a system of fields with an inverse limit that is also a field, denoted by $F$, then we can define the base change functor $\mathcal{A}(F)\rightarrow PP_{\mathcal{A}}(\mathscr{F})$ given by $V\mapsto \{V\otimes_F F_i\}_{i\in I}$. We say that an $\mathcal{A}$-patching problem for the system $\mathscr{F}$ has a solution if there exists an object $V\in\mathcal{A}(F)$ such that the patching problem is isomorphic to the patching problem obtained from $V$ under the base change functor.
\end{definition}

We now devote our attention to constructing a system of fields $\mathscr{F}$ such that the inverse limit of the system is a semi-global field and certain $\mathcal{A}$-patching problems over $\mathscr{F}$ have solutions.

\begin{notation}\label{modelsystem}
Let $T$ denote a complete discretely valued ring with residue field $k$ and fraction field $K$. Let $F$ be a finitely generated transcendence degree 1 field extension of $K$ and let $\widehat{X}$ be a regular $T$-model of $F$, a regular, connected, projective, $T$-curve with function field $F$. Such a model exists by \parencite{lipman1975introduction}, and we denote the closed fiber of $\widehat{X}$ by $X$. 

Given a finite set of closed points of $\widehat{X}$ on $X$ containing every point where irreducible components of $X$ meet, $\mathscr{P}$, we construct a system of fields in the following way. For each point $P\in \mathscr{P}$ we let $R_P$ be the local ring of $\widehat{X}$ at $P$ and $\widehat{R_P}$ denote the completion of $R_P$ at its maximal ideal. We define $F_P$ to be the fraction field of $\widehat{R_P}$.  

Next we let $\mathscr{W}$ be the set of all connected components of $X\backslash\mathscr{P}$. For every $U\in \mathscr{W}$ we let $R_U$ be the subring of $F$ of rational functions which are regular on $U$. We let $\widehat{R_U}$ be the $t$-adic completion of $R_U$. We define $F_U$ to be the fraction field of $\widehat{R_U}$. 

Lastly, we let $\mathscr{B}$ denote the set of all branches of $X$ at points in $\mathscr{P}$. Each element $\wp\in\mathscr{B}$ corresponds to a pair $(P, U)\in \mathscr{P}\times\mathscr{W}$ where $\wp$ is the branch of $X$ at $P$ lying on $U$. By abuse of notation, we let $\wp$ also denote the corresponding height 1 prime ideal containing $t$ in $\widehat{R_P}$. We define $\widehat{R_\wp}$ to be the completion of the localization of $\widehat{R_P}$ at $\wp$ with respect to its maximal ideal, and then define $F_\wp$ to be the fraction field of $\widehat{R_\wp}$.

In this setup we have a field $F_\xi$ for every $\xi\in \mathscr{I}=\mathscr{P}\cup\mathscr{W}\cup\mathscr{B}$ and we define a partial order on $\mathscr{I}$ by $P\succ\wp$ and $U\succ \wp$ whenever $\wp$ is the branch of $X$ at $P$ lying on $U$. Given such a branch it is clear that we have an inclusion of $F_P$ in $F_\wp$. Additionally, since $P$ is in the closure of $U$, and $U$ is irreducible, the radical ideal of $(t)$ in $R_U$ which we denote $\eta$ is prime and $(R_U)_\eta$ is canonically isomorphic to $(R_P)_\mathfrak{p}$ where $\mathfrak{p}=\wp\cap R_P$. This gives an inclusion $R_U\rightarrow \widehat{R_\wp}$ and since $\widehat{R_U}$ is the $t$-adic completion of $R_U$ (and $t\in\wp$ so $\widehat{R_\wp}$ is $t$-adically complete) this implies that $\widehat{R_U}$ is contained in $\widehat{R_\wp}$ and therefore $F_U$ includes into $F_\wp$. 

We call the system of fields given by such a triple $(\widehat{X},\mathscr{I},\mathscr{F})$ a semi-global patching framework for $F$.
\end{notation}

\begin{theorem}[Theorems 6.4 and 7.1 in \parencite{harbater2010patching}, and Proposition 3.3 in \parencite{harbater2015local}]\label{patchingequivalence}
Let $F$ be a semi-global field and $(\widehat{X},\mathscr{I},\mathscr{F})$ be a semi-global patching framework for $F$. For a field extension $L$ of $F$ let $\mathcal{A}(L)$ denote any of the following categories:
\begin{enumerate}
	\item[i.] the category of central simple algebras over $L$ with algebra homomorphisms, or
	\item[ii.] the category of $G$-Galois $L$-algebras for a fixed finite group $G$ (as defined on p.84 of \parencite{de2006separable}) with $G$-equivariant algebra homomorphisms.
\end{enumerate}
Then the base change functor $\mathcal{A}(F)\rightarrow PP_\mathcal{A}(\mathscr{F})$ is an equivalence of categories. In particular, every $\mathcal{A}$-patching problem has a unique solution up to isomorphism.
\end{theorem}

\begin{remark}
It will be useful to change a given model by blow-ups. We note that the blow-up of a regular model of $F$ at a closed point $P$ of the closed fiber is also a regular model of $F$ that has one extra connected component of the closed fiber, the exceptional divisor, which is isomorphic to $\mathbb{P}^1_{k(P)}$ where $k(P)$ is the residue field of $\mathscr{O}_{\widehat{X},P}$ (see Section 13.19 of \parencite{gortz2010algebraic} and Section 19.4 of \parencite{grothendieck1967etude}). This component intersects the rest of the closed fiber transversely.
\end{remark}

\section{Local Admissibility of $p$-Sylow Subgroup}\label{localad}

A key step in proving admissibility over a semi-global field using Theorem \ref{patchingequivalence} is establishing the admissibility of Sylow subgroups over points in the patching framework. For convenience, we record some definitions. 

\begin{definition}
Given a field $F$ and a finite field extension $E/F$, $E$ is called $F$-adequate if there exists a division algebra central over $F$ containing $E$ as a maximal subfield.
\end{definition}

\begin{definition}
A finite group $G$ is admissible over $F$ if there exists a Galois field extension $E/F$ with group $G$ such that $E$ is $F$-adequate.
\end{definition}

Harbater, Hartmann and Krashen proved the existence of certain adequate extensions over points in the patching framework in the case that $F$ is a function field over a complete discretely valued field with residue field $k$ where char$(k)$ does not divide the order of the group. 
 
\begin{lemma}[Lemma 4.3 and Proposition 4.4 in \parencite{harbater2011patching}]\label{sylowblocks}
Let $K$ be a complete discretely valued field with algebraically closed residue field $k$, and let F be a finitely generated field extension of K with transcendence degree 1. Let $(\widehat{X},\mathscr{I},\mathscr{F})$ be a semi-global patching framework for $F$ with a point $P\in \mathscr{P}$ such that $X$ is regular at $P$. Suppose $G$ is an abelian group of rank at most two with $\text{char}(k)\nmid \vert G\vert$, then there exists an $F_P$-adequate field extension $E/F_P$ with Galois group $G$ that splits over the unique branch of $X$ at $P$, i.e. $E\otimes_{F_P}F_\wp\cong F_\wp^{\vert G\vert}$ where $\wp$ is the unique branch of $X$ at $P$.
\end{lemma}

In this section we prove a statement similar to Lemma \ref{sylowblocks} in the case that $G$ is a $p$-group and $p$ is the characteristic of $k$. Unfortunately, we were not able to find a direct proof that the adequate field extension splits over the branch field, like the final statement proved in the lemma. In Section \ref{rotation} we will gain weaker control over the behavior of the adequate field extension over the branch field which will be enough to establish the compatibility of algebraic objects required for patching later on.

Working with $G$-Galois extensions in this ``bad" characteristic case is facilitated by being able to parameterize such extensions using results from generic Galois theory. We record some key results in the theory that will be used repeatedly. 

\begin{theorem}[Theorem 2.3 in \parencite{saltman1978noncrossed}]\label{genericgalois}
Let $p$ be prime and let $R(G)$ denote the polynomial ring $\mathbb{F}_p[x_1, x_2,\dots, x_m]$. If $G$ is a finite group of order $p^m$ then there exists a $G$-Galois $R(G)$-algebra, $S(G)$, such that:
\begin{enumerate}
	\item[i.] for every characteristic $p$ ring $R$ and ring homomorphism $f:R(G)\rightarrow R$, $S(G)\otimes_f R$ is a $G$-Galois $R$-algebra, and 
	\item[ii.] for every characteristic $p$ ring $R$ and $G$-Galois $R$-algebra $S$, there exists a ring homomorphism $f: R(G)\rightarrow R$, such that $S\cong S(G)\otimes_f R$.
\end{enumerate}
Here the subscript $f$ is a reminder that we view $R$ as an $R(G)$-module via the homomorphism $f$ in the tensor product.
\end{theorem}

\begin{lemma}[Lemma 5.6 in \parencite{saltman1982generic}]\label{krasners}
Let $F$ be a characteristic $p$ field, complete with respect to a nonarchimedean real-valued valuation $v$, and $G$ a group of order $p^m$ with $R(G)$ and $S(G)$ as in Theorem \ref{genericgalois}. Suppose $f:R(G)\rightarrow F$ is a ring homomorphism and let $a_i=f(x_i)$. There is a $\delta >0$ such that if $b_i\in F$ satisfy $\norm{b_i-a_i}_v<\delta$ for all $i$, then the ring homomorphism $f':R(G)\rightarrow F$ defined by setting $f'(x_i)=b_i$ has the property that $S(G)\otimes_f F\cong S(G)\otimes_{f'}F$ as $G$-Galois $F$-algebras. 
\end{lemma}

The following lemma was proved by Guico in \parencite{erneststhesis}. We present the statement and proof for completeness. 

\begin{lemma}[Lemma 3.2.1 in \parencite{erneststhesis}]\label{cyclicad}
Let $R$ be a domain with fraction field $F$ and suppose $\wp$ is a height 1 prime ideal of $R$ such that $\widehat{R_\wp}$, the completion of $R$ with respect to $\wp$, is a complete discretely valued ring. Let $\widehat{F}$ be the fraction field of $\widehat{R_\wp}$. If $E/F$ is a degree $n$ cyclic field extension of $F$ with group $G=\langle\sigma\rangle$ such that $\widehat{E}:=E\otimes_F \widehat{F}$ is an unramified field extension of $\widehat{F}$, then $E$ is $F$-adequate. 
\end{lemma}

\begin{proof}
We let $b$ be a uniformizer of $\widehat{F}$ and $v$ denote the discrete valuation on $\widehat{F}$. It suffices to show that the cyclic algebra $A=(E,\sigma, b)$ is a division algebra. We will require some standard facts about the period and index of cyclic algebras which can be found in Sections 12-15 of \parencite{pierce1982associative}.

Consider $\widehat{A}:=A\otimes_F \widehat{F}=(\widehat{E},\sigma, b)$, a cyclic algebra over $\widehat{F}$. The period of $\widehat{A}$ is equal to the order of $\overline{b}$ in $\widehat{F}^\times/\text{N}_{\widehat{E}/\widehat{F}}(\widehat{E^\times})$. It suffices to show that $n$ divides the order of $\overline{b}$ because then 
\[n\mid \text{per}(\widehat{A})\mid \text{per}(A)\mid \text{ind}(A)\mid \deg{A}=n,\]
showing that $\text{ind}(A)=\text{deg}(A)$ and $A$ is a division algebra. 

Since $\widehat{E}$ is an unramified degree $n$ extension of the complete discretely valued field $\widehat{F}$, $v$ extends uniquely to a discrete valuation $w$ on $\widehat{E}$ given by $w(x):=\frac{1}{n}v(\text{N}_{\widehat{E}/\widehat{F}}(x))$ for all $x\in\widehat{E}$ (Corollary II.2.4 in \parencite{serre2013local}). We have a group homomorphism from $F^\times$ to $\mathbb{Z}/n\mathbb{Z}$ given by $x\mapsto \overline{v(x)}$ and the previous sentence shows that $\text{N}_{\widehat{E}/\widehat{F}}(E^\times)$ is in the kernel of this map, so the homomorphism factors through $\widehat{F}^\times/\text{N}_{\widehat{E}/\widehat{F}}(\widehat{E^\times})$. Under this map it is clear that $\overline{b}\mapsto \overline{v(b)}=\overline{1}$ and since $\overline{1}$ generates $\mathbb{Z}/n\mathbb{Z}$, $n$ divides the order of $\overline{b}$. 
\end{proof}

The last lemma in this section is a slightly weaker version of Lemma 3.2.2 in \parencite{erneststhesis}. A great deal of the proof is the same as that given by Guico. We streamline parts of the argument since we do not require constructions as explicit as those given in the original proof. 

\begin{lemma}\label{pexistence}
Let $R$ be an equicharacteristic complete regular local ring of dimension 2 with residue field $k$ of characteristic $p>0$. If $F$ is the fraction field of $R$ and $G$ is a group with $\vert G\vert=p^m$ for some $m\geq 1$, then $G$ is admissible over $F$. 
\end{lemma}

\begin{proof}
By the Cohen structure theorem it suffices to prove the result for $R=k[[x,y]]$ and $F=k((x,y))$. We note that $\wp=(x)$ is a height 1 prime ideal of $R$, and $\widehat{R_\wp}=k((y))[[x]]$ is a complete discretely valued ring with fraction field $\widehat{F}=k((y))((x))$ and residue field $k((y))$. 

We begin by proving the statement in the case that $G$ is cyclic. By Lemma \ref{cyclicad} it is enough to construct a Galois extension $E/F$ with group $G$ such that $\widehat{E}:=E\otimes_F \widehat{F}$ is an unramified field extension of $\widehat{F}$. We note that since $G$ is cyclic and char$(k((y)))=p$, a Galois extension of the residue field $k((y))$ with group $G$ exists if $\mathcal{P}(k((y)))\neq k((y))$, where $\mathcal{P}(z):=z^p-z$ (Theorem 8.32 in 8.11 of \parencite{jacobsonbasic89}). Let $v$ denote the $y$-adic valuation on $k((y))$. If $y^{-1}\in\mathcal{P}(k((y)))$ then there exists $z\in k((y))$ such that $v(z^p-z)=v(y^{-1})=-1$, so \[-1=v(z^p-z)\geq \min\{pv(z),v(z)\}\] with equality holding if and only if $pv(z)\neq v(z)$. Either case yields a contradiction, so $y^{-1}\notin\mathcal{P}(k((y)))$. Therefore there exists a Galois field extension of $k((y))$ with group $G$.

Since $k((y))$ is the residue field of the complete discretely valued field $\widehat{F}$, there exists a corresponding unramified extension $\widehat{E_\wp}$ of $\widehat{F}$ (Theorem III.5.2 in \parencite{serre2013local}) and $\widehat{E_\wp}/\widehat{F}$ is Galois with group $G$. Since $\widehat{E_\wp}$ is a $G$-Galois $\widehat{F}$-algebra and char$(\widehat{F})=p$ there exists a homomorphism $f:R(G)\rightarrow \widehat{F}$ such that $\widehat{E_\wp}\cong S(G)\otimes_f \widehat{F}$ with $R(G)$ and $S(G)$ as in Theorem \ref{genericgalois}. Additionally, there exists a $\delta>0$ as in Lemma \ref{krasners} and since $F$ is dense in $\widehat{F}$, for every $1\leq i\leq m$ there exist $b_i\in F$ such that $\norm{b_i-f(x_i)}_v<\delta$. Since char$(F)=p$ we can define a homomorphism $f':R(G)\rightarrow F$ by $x_i\mapsto b_i$ and set $E:=S(G)\otimes_{f'} F$. By Theorem \ref{genericgalois} we know that $E$ is a $G$-Galois $F$-algebra, and Lemma \ref{krasners} implies that $\widehat{E}=E\otimes_F \widehat{F}=S(G)\otimes_{f'}\widehat{F}\cong \widehat{E_\wp}$. Since $\widehat{E_\wp}$ is a field, $\widehat{E}$ is a field as well. This completes the proof in the case that $G$ is cyclic. 

Now let $G$ be any group of order $p^m$. Since we know there exists a division algebra $D$ over $F$ with a maximal subfield that is Galois over $F$ with group $C_{p^m}$, Theorem 1' of \parencite{saltman1977splittings} implies that $D$ also has a maximal subfield that is Galois over F with Galois group $G$. The additional condition that $F$ must satisfy in the statement of Theorem 1' of \parencite{saltman1977splittings}, that $\text{dim}_{\mathbb{F}_p}(F/\mathcal{P}(F))$ must be infinite, is automatically satisfied as a consequence of Theorem 3 in \parencite{saltman1977splittings} and the fact that $D\neq F$.
\end{proof}

\section{Branch Behavior of $p$-Sylow Extension}\label{rotation}

In this section we prove some technical lemmas which will allow us to control the behavior of division algebras and maximal subfields such as those shown to exist in Section \ref{localad} over the branch fields in our patching framework. As we are interested in $p$-power degree algebras in characteristic $p$, we will use the homological different to substitute for standard ramification and residue map techniques.

\begin{definition}
Let $R$ be a commutative ring and $A$ be an $R$-algebra. Let $A^o$ denote the opposite algebra of $A$ and set $A^e:=A\otimes_R A^o$. We have the following homomorphism of left $A^e$-modules (extended linearly):
\begin{align*}
\mu: A^e&\rightarrow A\\
a\otimes b^o&\mapsto ab\\ 
\end{align*}
Letting $J$ denote the kernel of this map we get the following short exact sequence of left $A^e$-modules:
\[
\begin{tikzcd}
    0 \arrow{r} & J \arrow{r} & A\otimes_R A^{o} \arrow{r}{\mu}  & A \arrow{r}  & 0 \\
\end{tikzcd}
\]
The homological different of the $R$-algebra $A$ is the image of the right annihilator of $J$ in $A$ under $\mu$, i.e. 
\[\mathfrak{H}(A/R):=\mu(Ann_{A^e}(J)).\]
Using the fact that $\{(a\otimes 1^o-1\otimes a^o):a\in A\}$ is a generating set for $J$, it can be checked that $\mathfrak{H}(A/R)$ is an ideal of the center of $A$ directly from the definition. In the case that $A$ is central over $R$ we observe that $\mathfrak{H}(A/R)=R$ if and only if there exists a separability idempotent for $A$ (see \parencite{de2006separable} section II.1). So when $A$ is central over $R$, $A$ is a separable $R$-algebra if and only if $\mathfrak{H}(A/R)=R$.
\end{definition}

\subsubsection*{Field Extension Unramified Almost Everywhere}
Now we study the ramification of field extensions over branch fields in the patching framework. 
\begin{definition}\label{ramification}
 Let $R$ be a regular local ring with fraction field $F$, and let $E/F$ be a finite separable extension. Let S be the integral closure of $R$ in $E$, and let $\wp$ be a prime ideal of $R$. We say that $E$ is unramified at $\wp$ if the induced map $R_{\wp}\rightarrow S_{q}$ is an unramified local homomorphism of local rings for every prime $q$ of S lying over $\wp$. Otherwise, we say that $E$ is ramified at $\wp$.
\end{definition}

\begin{lemma}\label{unramified_branches}
Let $R$, $S$, $F$, and $E$ be as in Definition \ref{ramification}. Additionally, assume that $R$ has Krull dimension 2. Then $E$ is only ramified at finitely many height 1 prime ideals of $R$.
\end{lemma}

\begin{proof}
By Theorem B.11 in \parencite{leuschke2012cohen}, $E$ is ramified at a prime ideal, $\wp$, of $R$ if and only if there exists a prime ideal, $q$, of $S$ lying over $\wp$ such that the homological different $\mathfrak{H}(S/R)$ is contained in $q$. We want to show there are only finitely many such prime ideals of $R$ of height 1. 

Let $q$ be a prime ideal of $S$ lying over a height 1 prime ideal of R, $\wp$, and suppose $q$ contains $\mathfrak{H}(S/R)$. Suppose, for the sake of contradiction, that $q$ is not an associated prime of $\mathfrak{H}(S/R)$, then it must strictly contain a prime ideal minimal over $\mathfrak{H}(S/R)$, $q'$, so we have $\mathfrak{H}(S/R)\subseteq q'\subsetneq q$. This implies that $\wp':=q'\cap R\subsetneq q\cap R=\wp$. Since $\wp'$ is a prime ideal of $R$ strictly contained in a height 1 prime ideal, and  $R$ is a domain, this implies that $\wp'$ is the zero ideal. Since E is an algebraic extension of F, this implies that $q'$ is the zero ideal of $S$. However, since $R$ is regular and $S$ is normal, Theorem B.12 in \parencite{leuschke2012cohen} implies that $\mathfrak{H}(S/R)$ is an ideal of pure height 1. So $q'$ must be a height 1 prime in $S$, which is a contradiction. Therefore, $q$ must be an associated prime of $\mathfrak{H}(S/R)$.

Since $S$ is Noetherian, $\mathfrak{H}(S/R)$ has only finitely many associated primes. Let $Q=\{q_1, q_2, \dots q_n\}$ be the subset of the set of associated primes of $\mathfrak{H}(S/R)$ consisting of all associated primes of $\mathfrak{H}(S/R)$ such that $\wp_i=q_i\cap R$ is a height 1 prime of $R$. Then $\{\wp_1, \wp_2, \dots, \wp_n\}$ is the set of all height 1 prime ideals of $R$ at which $E$ is ramified. 
\end{proof}

\subsubsection*{Division Algebra Splits Almost Everywhere} 
Similarly, the next lemma proves that if $R$ is an equicharacteristic complete regular local ring of dimension 2 with fraction field $F$ and algebraically closed residue field, and $D$ is a division algebra central over $F$, then $D$ splits over all but finitely many height 1 primes of $R$. We say that $D$ splits over a height 1 prime ideal $\wp$ of $R$ if $D\otimes_F F_\wp$ is trivial in the Brauer group of $F_\wp$, $Br(F_\wp)$, where $F_\wp$ is the fraction field of $\widehat{R_\wp}$. 

\begin{definition}\label{order}
Let $R$ be an integrally closed Noetherian domain, $F$ be the fraction field of $R$, and $A$ be a central simple algebra over $F$. We will call a subring $\Lambda$ of $A$ an $R$-order in $A$ if $\Lambda$ is a finitely generated torsion-free $R$-module, and the $F$-span of a set of $R$-generators of $\Lambda$ is $A$.
\end{definition}

\begin{lemma}\label{split_branches}
Let $R$ be an equicharacteristic complete regular local ring of dimension 2 with algebraically closed residue field. Let $F$ be the fraction field of $R$, and $D$ be a division algebra central over $F$. Then $D$ splits over all but finitely many height 1 prime ideals of $R$. 
\end{lemma}

\begin{proof}
First we note that $D$ contains an $R$-order, $\Lambda$, central over $R$ (see p.109 of \parencite{reiner2003maximal} for an explicit construction). We will begin by considering the separability of $\Lambda$ over height 1 prime ideals of $R$ and later show the connection to the splitting behavior of $D$. Given a prime ideal $\wp$ of $R$, $\Lambda\otimes_R R_\wp$ is separable over $R_\wp$ if and only if $\wp$ does not contain $\mathfrak{H}(\Lambda/R)$. This is due to Proposition 4.4 of \parencite{auslander1960brauer} which implies that $\mathfrak{H}(\Lambda\otimes_R R_\wp/R_\wp)=\mathfrak{H}(\Lambda/R)\otimes_R R_\wp$.

We first show that only finitely many height 1 prime ideals of $R$ can contain $\mathfrak{H}(\Lambda/R)$. Each height 1 prime ideal must either be an associated prime of $\mathfrak{H}(\Lambda/R)$, or strictly contain a minimal prime of $\mathfrak{H}(\Lambda/R)$. However, the zero ideal of $R$ is the only prime ideal strictly contained in a height 1 prime ideal. Therefore, as long as $\mathfrak{H}(\Lambda/R)$ is non-zero, the height 1 prime ideals containing $\mathfrak{H}(\Lambda/R)$ are all associated primes of $\mathfrak{H}(\Lambda/R)$ and there are only finitely many associated primes of $\mathfrak{H}(\Lambda/R)$ since R is Noetherian.

We see that $\mathfrak{H}(\Lambda/R)$ is not zero because $\mathfrak{H}(\Lambda\otimes_R F/F)=\mathfrak{H}(\Lambda/R)\otimes_R F$ (Proposition 4.4 in \parencite{auslander1960brauer}). Since $\Lambda\otimes_R F=D$, and $D$ is a separable $F$-algebra, 
\[\mathfrak{H}(\Lambda/R)\otimes_R F=\mathfrak{H}(\Lambda\otimes_R F/F)=\mathfrak{H}(D/F)=F.\]
Therefore, $\mathfrak{H}(\Lambda/R)\otimes_R F\neq 0$.

Now suppose $\wp$ is a height 1 prime ideal that does not contain $\mathfrak{H}(\Lambda/R)$. Then $\Lambda_\wp$ is a central separable $R_\wp$-algebra, which implies that $\widehat{\Lambda_\wp}:=\Lambda_\wp\otimes_{R_\wp}\widehat{R_\wp}=\Lambda\otimes_R \widehat{R_\wp}$ is a central separable $\widehat{R_\wp}$-algebra (see Lemma 2.5.1 of \parencite{de2006separable}). In this case the class of $\widehat{\Lambda_\wp}$, denoted $[\widehat{\Lambda_\wp}]$, is an element of the Brauer group of $\widehat{R_\wp}$. Since $\widehat{R_\wp}$ is a complete local ring, the canonical map from the Brauer group of $\widehat{R_\wp}$ to the Brauer group of its residue field is injective (Proposition IV.1.6 in \parencite{milne2016etale} ). Therefore we examine the residue field of $\widehat{R_\wp}$ and its Brauer group.\\
\indent By the Cohen structure theorem, $R\cong k[[x,y]]$ where $k$ is the residue field of $R$ and we can assume that the isomorphism maps $x$ to a generator of $\wp$, which is principal. So $R_\wp\cong k[[x,y]]_{(x)}$ and $\widehat{R_\wp}\cong k((y))[[x]]$. This implies that the residue field of $\widehat{R_\wp}$ is isomorphic to $k((y))$, and since $k$ is algebraically closed by assumption, $k((y))$ is a $C_1$-field. Therefore, the Brauer group of the residue field of $\widehat{R_\wp}$ is trivial (see Proposition 6.2.3 and Theorem 6.2.11 in \parencite{gille2017central}). As a result, Br$(\widehat{R_\wp})$ is trivial as well.

Since $\widehat{\Lambda_\wp}\otimes_{\widehat{R_\wp}} F_\wp\cong D\otimes_F F_\wp$, where $F_\wp$ is the fraction field of $\widehat{R_\wp}$, the class of $D\otimes_F F_\wp$ in Br$(F_\wp)$ is the image of the class of $\widehat{\Lambda_\wp}$ in Br$(\widehat{R_\wp})$ under the map induced by the inclusion of $\widehat{R_\wp}$ in $F_\wp$. Therefore $D$ splits over every height one prime ideal $\wp$ where the class of $\widehat{\Lambda_\wp}$ is the trivial element of Br$(\widehat{R_\wp})$, i.e. all the height 1 primes that do not contain the homological different, $\mathfrak{H}(\Lambda/R)$. So $D$ splits over all but finitely many height 1 prime ideals of $R$. 
\end{proof}

Now we show that a ring such as the one in the previous lemma has infinitely many distinct height 1 prime ideals. 

\begin{lemma}\label{R_structure}
Let $R$ be an equicharacteristic complete regular local ring of dimension 2. Then $R$ has infinitely many height 1 prime ideals.
\end{lemma}

\begin{proof}
Again, by the Cohen structure theorem, $R\cong k[[x,y]]$ where $k$ is the residue field of $R$ and the isomorphism can be chosen to send $x$ and $y$ to any given regular system of parameters in $R$. Therefore, we can just show that $k[[x,y]]$ has infinitely many height 1 prime ideals.\\
\indent We note that $m=(x,y)$ is the maximal ideal of $k[[x,y]]$ and it is clear that for any $n\in\mathbb{Z}_{\geq1}$, the ideal $(x+y^n,y)$ is also equal to $m$. Since this is a minimal way to generate the maximal ideal, we can see that for each $n$, $x+y^n$ is a regular parameter and therefore $(x+y^n)$ is a height 1 prime of $k[[x,y]]$. Since $(x+y^i)\neq(x+y^j)$ when $i\neq j$, these height 1 primes are distinct and there are infinitely many of them.
\end{proof}

\subsubsection*{Changing Coordinates}
The following proposition is the main result of this section which we will use to control the behavior of division algebras and maximal subfields such as those shown to exist in Section \ref{localad} over a chosen height 1 prime ideal. To obtain good behavior over a specific height 1 prime ideal we leverage the Cohen structure theorem and make a change of coordinates to transport good behavior from an arbitrary height 1 prime ideal. 

\begin{proposition}\label{rotationprop}
Let $R$ be an equicharacteristic complete regular local ring of dimension 2 with algebraically closed residue field, $F$ be the fraction field of $R$, and $D$ be a division algebra central over $F$ with a maximal subfield $E$ separable over $F$. Then, given a height 1 prime ideal $\wp$ of $R$, there exists a finite dimensional division algebra $D'$ over $F$ with maximal subfield $E'$ isomorphic to $E$ such that $E'$ is unramified at $\wp$ and $D'$ splits over $\wp$. 
\end{proposition}
\begin{proof}
First we note that if $E$ is unramified at $\wp$ and $D$ splits over $\wp$ then we're done. So we suppose now that this is not the case. By Lemma \ref{unramified_branches}, $E$ is unramified over all but finitely many height 1 prime ideals of $R$, and by Lemma \ref{split_branches}, $D$ is split over all but finitely many height 1 primes ideals of $R$. Since $R$ has infinitely many height 1 prime ideals, there exists a height 1 prime ideal of $R$, $\tilde{\wp}\neq\wp$, such that $E$ is unramified at $\tilde{\wp}$ and $D$ splits over $\tilde{\wp}$ simultaneously.

As in previous proofs we know there exists a ring isomorphism $\psi:k[[x,y]]\rightarrow R$ which maps $x$ to a generator of $\tilde{\wp}$. Similarly, there is an isomorphism $\psi':k[[x,y]]\rightarrow R$ which maps $x$ to a generator of $\wp$ and so by composing $\psi' \circ \psi^{-1}$ we obtain an automorphism of $R$ which maps a generator of $\tilde{\wp}$ to a generator of $\wp$.

This automorphism can be extended to an automorphism $\phi$ of $F$. We consider the algebra $D'=D\otimes_\phi F$. Here we regard $D'$ as an $F$-algebra via $\phi$, so the action of $F$ on $D'$ is defined by $c\cdot (x\otimes d)=x\otimes cd=x\phi^{-1}(c)\otimes d$ for $c,d \in F$ and $x\in D$. We note that $D'$ is isomorphic to $D$ as a ring via the map $d\mapsto d\otimes 1$ therefore it is still a division ring with center $F$. For this same reason $D'$ also has a maximal subfield isomorphic to $E$, $E'=E\otimes_\phi F\subset D'$.

However $D$ and $D'$ are not necessarily isomorphic as $F$-algebras, and the ring isomorphism described above does not respect the $F$-algebra structure. We show that $D'$ is a division algebra central over $F$ which splits over $\wp$ and that $E'$ is unramified at $\wp$.

We let $\varphi: E\rightarrow E'$ be the ring isomorphism defined by $e\mapsto e\otimes 1$ but consider $R$ and $F$ as subrings of $E'$ via the injection $c\mapsto 1\otimes c$ for all $c\in F$. We note that as subsets of $E'$ these are equal to $\varphi(R)$ and $\varphi(F)$ respectively. Therefore, the integral closure of $R$ in $E'$ is the integral closure of $\varphi(R)$ in $E'$ which is just $S'=\varphi(S)$ where $S$ is the integral closure of $R$ in $E$. Now we consider $\wp$ in $E'$ where we think of it inside $\varphi(R)$ as the set \[\{1\otimes p \mid p\in\wp\}=\{\phi^{-1}(p)\otimes1\mid p\in\wp\}=\{p'\otimes 1\mid p'\in\tilde{\wp}\}=\varphi(\tilde{\wp}),\] since $\phi^{-1}(\wp)=\tilde{\wp}$. So for any prime ideal $q'$ of $S'$ lying over $\wp$ we have $q'=\varphi(q)$ for some prime $q$ of $S$ lying over $\tilde{\wp}$. Then the map $R_\wp\rightarrow S'_{q'}$ is just the composition of $\varphi$ with the map $R_{\tilde{\wp}}\rightarrow S_q$, and so is an unramified local homomorphism of local rings. Therefore $E'$ is unramified at $\wp$.

Now we consider the class $D'\otimes F_\wp$ in Br$(F_\wp)$. We note first that the original automorphism of $R$ can be extended to an isomorphism $R_{\tilde{\wp}}\rightarrow R_\wp$ and this induces an isomorphism of the completions. Since the completions are domains, this can be extended to an isomorphism of the fraction fields and so we have an isomorphism $\overline{\phi}:F_{\tilde{\wp}}\rightarrow F_\wp$, which when restricted to $F$ considered as a subfield of each of these fields is just the automorphism $\phi$. Therefore, we have an isomorphism of Brauer groups Br$(F_{\tilde{\wp}})\rightarrow\text{Br}(F_\wp)$ given by $[A]\mapsto[A\otimes_{\overline{\phi}}F_\wp]$. Now we note that $D'\otimes_F F_\wp=D\otimes_\phi F\otimes_F F_\wp\cong D\otimes_F F_{\tilde{\wp}}\otimes_{\overline{\phi}} F_\wp$ as $F_\wp$- algebras via the map $d\otimes c\otimes e\mapsto d\otimes\phi^{-1}(c)\otimes e$. Since $[D\otimes_F F_{\tilde{\wp}}\otimes_{\overline{\phi}} F_\wp]$ is the image of $[D\otimes_F F_{\tilde{\wp}}]$ under the isomorphism of the Brauer groups described above, and $[D\otimes_F F_{\tilde{\wp}}]$ is trivial in Br$(F_{\tilde{\wp}})$, it's image must be trivial in Br$(F_\wp)$. Therefore, $D'$ splits over $\wp$. 
\end{proof} 

To summarize, Lemma \ref{pexistence} and Proposition \ref{rotationprop} imply the following proposition.

\begin{proposition}\label{specialcase}
Let $R$ be an equicharacteristic complete regular local ring of dimension 2 with algebraically closed residue field $k$ of characteristic $p>0$, and let $F$ be the fraction field of $R$. Suppose that $G$ is a finite p-group. Given a height 1 prime ideal $\wp$ of $R$, there exists a finite dimensional division algebra $D$ over $F$ with maximal subfield $E$ such that:
\begin{enumerate} 
	\item[i.]$E/F$ is Galois with group $G$,
	\item[ii.] $E$ is unramified at $\wp$, and
	\item[iii.] $D$ splits over $\wp$.
\end{enumerate}
\end{proposition}

\section{Establishing Compatibility}\label{compatibilitysection}

In this section we will describe how the local pieces constructed in the previous two sections can be put together to get a portion of a patching problem over a semi-global field. Throughout this section $T$ denotes a complete discrete valuation ring with uniformizer $t$, $K$ is the fraction field of $T$, and $F$ is a transcendence degree 1 field extension of $K$. Furthermore, $\text{char}(K)=\text{char}(k)=p>0$ where $k$ is the residue field of $T$.

Additionally, we suppose that there exists a semi-global patching framework $(\widehat{X},\mathscr{I},\mathscr{F})$ of $F$ as in Notation \ref{modelsystem} such that $\mathscr{P}$ contains at least two points, $P$ and $P'$, that satisfy the following conditions:
\begin{enumerate}
	\item[i.] one of the connected components of $X\backslash\mathscr{P}$ is an affine open set $U$ with $\overline{U}=U\cup\{P,P'\}$,
	\item[ii.] $U$ is the unique element of $\mathscr{W}$ such that $P\in\overline{U}$, and
	\item[iii.] there exists a different connected component of $X\backslash\mathscr{P}$, $U'$, such that $P'\in\overline{U'}$.
\end{enumerate}

With this patching framework there is a unique branch of $X$ at $P$, and we denote this branch by $\wp$. At $P'$ we have at least two branches of $X$; we denote the branch lying on $\overline{U}$ (resp. $\overline{U'}$) by $\wp_U$ (resp. $\wp_{U'}$). The containments of the fields $\{F_\xi\}_{\xi\in\{P,P',U,U',\wp,\wp_U,\wp_{U'}\}}$, as described in Notation \ref{modelsystem}, is summarized by the following diagram:

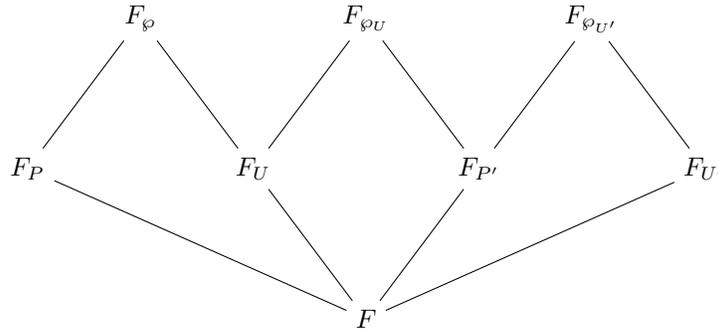
\begin{center}
\begin{tikzpicture}

    \node (Q1) at (0,0) {$F$};
    \node (Q2) at (-4.5,2) {$F_P$};
    \node (Q3) at (-1.5,2) {$F_U$};
    \node (Q4) at (1.5,2) {$F_{P'}$};
    \node (Q5) at (4.5,2) {$F_{U'}$};
    \node (Q6) at (-3,4) {$F_\wp$};
    \node (Q7) at (0,4) {$F_{\wp_U}$};
    \node (Q8) at (3,4) {$F_{\wp_{U'}}$};

    \draw (Q1)--(Q2);
    \draw (Q1)--(Q3);
    \draw (Q1)--(Q4);
    \draw (Q1)--(Q5);
    \draw (Q2)--(Q6);
    \draw (Q3)--(Q6);
    \draw (Q3)--(Q7);
    \draw (Q4)--(Q7);
    \draw (Q4)--(Q8);
    \draw (Q5)--(Q8);

\end{tikzpicture}
\end{center}

Given a division algebra over $F_P$ with a maximal subfield Galois over $F_P$ exhibiting good behavior over $F_\wp$ we will construct, in a series of steps detailed below, compatible algebraic structures over each of the remaining fields $F_\xi$ for $\xi\in\{P',U,U',\wp_U,\wp_{U'}\}$, working our way from left to right across the diagram. The existence of these compatible structures will be required in the proof of the main result as a portion of the data that makes up the patching problem.

First we record an auxiliary lemma that we will use a couple of times in this section. 

\begin{lemma}\label{restriction}
Let $A$ and $A'$ be central simple algebras over a field $F$ with $A\cong A'$ and $G$ be a finite group with $\text{deg}(A)=\vert G\vert$. Suppose we have $G$-Galois $F$-subalgebras of $A$ and $A'$ denoted $E$ and $E'$ respectively. If there is an isomorphism $\phi:E\rightarrow E'$ then there exists an isomorphism $\psi:A\rightarrow A'$ that restricts to $\phi$ on $E$.  
\end{lemma}

\begin{proof}
There is some isomorphism $\psi':A\rightarrow A'$ and we will show that $\psi'$ can be adjusted to obtain an isomorphism which restricts to $\phi$. First we show that a $G$-Galois $F$-algebra is a Frobenius algebra as defined on page 42 of \parencite{jacobson2009finite}. We have a linear map $l:E\rightarrow F$ given by $l(x)=\sum_{\sigma\in G}\sigma(x)$, and the corresponding bilinear map $l(x,y):=l(xy)$. This map is non-degenerate since there exist $x_1,x_2,\dots x_n, y_1,y_2,\dots y_n \in E$ such that for every $\sigma\in G$, $\sum_{i=1}^{n} \sigma(x_i)y_i = \delta_{\sigma,1}$ (see Proposition III.1.2 in \parencite{de2006separable}). So if $x\in E$ such that $l(x,y)=0$ for all $y\in E$, then 
\begin{align*}
x&=x(\sum_{i=1}^n x_i y_i)=x(\sum_{i=1}^n x_i y_i)+\sum_{\sigma\neq 1}\sigma(x)(\sum_{i=1}^{n} \sigma(x_i)y_i)\\
&=\sum_{i=1}^{n}(\sum_{\sigma\in G}\sigma(x x_i))y_i=\sum_{i=1}^{n} l(x,x_i)y_i=0.
\end{align*}
Since this bilinear form is non-degenerate, we know that the kernel of $l$ is a hyperplane of $E$. Additionally, $\text{ker}(l)$ contains no nontrivial ideals. Suppose for instance that $I$ is an ideal of $E$ contained in $\text{ker}(l)$, then if $x\in I$, for any $y\in E$, $xy\in I\subset \text{ker}(l)$ so $l(x,y)=l(xy)=0$ and so $x=0$.

Similarly, $\psi'(E)$ is also Frobenius and we see that $\phi\circ\psi'^{-1}$ defines an isomorphism of $\psi'(E)$ into $A'$. Since $\text{dim}(\psi'(E))=\vert G\vert=\text{deg}(A')$ Theorems 2.2.2 and 2.2.3 in \parencite{jacobson2009finite} imply that this isomorphism can be extended to an inner automorphism of $A'$, let $\alpha:A'\rightarrow A'$ denote this inner automorphism. Then we see that $\psi:=\alpha\circ\psi':A\rightarrow A'$ is an isomorphism such that the following diagram commutes: 
\[
	\begin{tikzcd}
E \arrow{r}{\phi} \arrow[swap]{d} & E' \arrow{d} \\
A \arrow{r}{\psi} & A'.
	\end{tikzcd}
\] 
\end{proof}

\begin{lemma}\label{compatibility1}
Suppose $D$ is a division algebra over $F_P$ with maximal subfield $E$ such that the following conditions are satisfied:
\begin{enumerate}
\item[i.] $E/F$ is a Galois field extension with $Gal(E/F)=G$ where $\vert G\vert=p^m$,
\item[ii.] E is unramified over $\wp$, and
\item[iii.] D splits over $F_\wp$.
\end{enumerate}

Then there exists a central simple algebra over $F_U$, $A_U$, containing a $G$-Galois $F_U$-subalgebra, $E_U$, such that $A_U\otimes_{F_U}F_\wp\cong D\otimes_{F_P}F_\wp$ as $F_\wp$-algebras and the isomorphism restricts to an isomorphism of $G$-Galois $F_\wp$-algebras, $E_U\otimes_{F_U}F_\wp\cong E\otimes_{F_P}F_\wp$.
\end{lemma}

\begin{proof}
We begin by constructing $E_U$. We note that 
$$E\otimes_{F_P}F_\wp=\widehat{E_{q_1}}\times\widehat{E_{q_2}}\times\dots\times\widehat{E_{q_l}},$$
where $\widehat{E_{q_i}}$ is the completion of $E$ with respect to the valuation corresponding to a prime $q_i$ lying over $\wp$ in $E$ and $\{q_1, q_2, \dots q_l\}$ is the set of all such prime ideals. This is explained in chapter 18 of \parencite{pierce1982associative}. We let $B_i$ be the integral closure of $\widehat{R_\wp}$ in $\widehat{E_{q_i}}$ and set $B_\wp:=B_1\times B_2\times\dots\times B_l$. We note that $B_\wp$ is a $G$-Galois $\widehat{R_\wp}$-algebra as a result of assumptions (i) and (ii). Additionally, $E\otimes_{F_P}F_\wp\cong B_\wp\otimes_{\widehat{R_\wp}}F_\wp$. 

Since the inclusion of $F_U$ in $F_\wp$ is the extension of the canonical inclusion $\widehat{R_U}\hookrightarrow \widehat{R_\wp}$, it is enough to construct a $G$-Galois $\widehat{R_U}$-algebra, $B_U$, such that $B_U\otimes_{\widehat{R_U}}\widehat{R_\wp}\cong B_\wp$. Then we can set $E_U:=B_U\otimes_{\widehat{R_U}}F_U$.

Following the notation of Saltman, let $R(G)$ denote $\mathbb{F}_p[x_1,x_2,\ldots,x_m]$. Since $G$ is a finite $p$-group of order $p^m$, and char$(\widehat{R_\wp})=p$, there exists a generic $G$-Galois $R(G)$-algebra, $S(G)$, and a ring homomorphism $f_{B_\wp}: R(G)\rightarrow \widehat{R_\wp}$ such that $B_\wp\cong S(G)\otimes_{f_{B_\wp}} \widehat{R_\wp}$ as $G$-Galois $\widehat{R_\wp}$-algebras (Theorem \ref{genericgalois}). We can compose $f_{B_\wp}$ with the quotient map from $\widehat{R_\wp}$ to the residue field of $\widehat{R_\wp}$, $k(\wp)$, to obtain a $G$-Galois $k(\wp)$-algebra$, S(G)\otimes_{\overline{f_{B_\wp}}}k(\wp)$. Setting $\overline{B_\wp}:=B_\wp\otimes_{\widehat{R_\wp}}k(\wp)$ we see that,
\[S(G)\otimes_{\overline{f_{B_\wp}}}k(\wp)\cong (S(G)\otimes_{f_{B_\wp}} \widehat{R_\wp})\otimes_{\widehat{R_\wp}}k(\wp)\cong B_\wp\otimes_{\widehat{R_\wp}}k(\wp)=\overline{B_\wp}.\] 

Now, $k(\wp)$ is a characteristic $p$ field complete with respect to a nonarchimedean real-valued valuation $v$. So by Lemma \ref{krasners} there exists a $\delta >0$ such that if $b_i\in k(\wp)$ satisfy $\norm{b_i-\overline{f_{B_\wp}}(x_i)}_v<\delta$, then the new ring homomorphism $f':R(G)\rightarrow k(\wp)$ defined by setting $f'(x_i)=b_i$ has the property that $S(G)\otimes_{\overline{f_{B_\wp}}} k(\wp)\cong S(G)\otimes_{f'} k(\wp)$ as $G$-Galois $k(\wp)$-algebras. 

If we let $\eta$ denote the radical ideal of $(t)$ in $\widehat{R_U}$, we note that $\widehat{R_U}/\eta$ is dense in $k(\wp)$ (see Notation 2.2 in \parencite{harbater2019local}), so there exist $\overline{b_i}\in k(\wp)\cap\widehat{R_U}/\eta$ such that $\norm{\overline{b_i}-f_{\overline{B_\wp}}(x_i)}<\delta$ and so there exist a homomorphism $f'_{\overline{B_\wp}}:R(G)\rightarrow k(\wp)$ so that $S(G)\otimes_{f'_{\overline{B_\wp}}}k(\wp)\cong S(G)\otimes_{f_{\overline{B_\wp}}}k(\wp)\cong \overline{B_\wp}$ as $G$-Galois $k(\wp)$-algebras.

For each $\overline{b_i}\in \widehat{R_U}/\eta$ we pick a representative $b_i$ in $\widehat{R_U}$ and define a ring homomorphism $f_U:R(G)\rightarrow \widehat{R_U}$ by mapping each $x_i$ to $b_i$. Therefore we have the $G$-Galois $\widehat{R_U}$-algebra, $B_U:=S(G)\otimes_{f_U}\widehat{R_U}$. By construction, $(B_U\otimes_{\widehat{R_U}}\widehat{R_U}/\eta)\otimes_{\widehat{R_U}/\eta}k(\wp)\cong \overline{B_\wp}$. Additionally, $(B_U\otimes_{\widehat{R_U}}\widehat{R_U}/\eta)\otimes_{\widehat{R_U}/\eta}k(\wp)\cong (B_U\otimes_{\widehat{R_U}}\widehat{R_\wp})\otimes_{\widehat{R_\wp}}k(\wp)$ because the following square commutes:

\[
	\begin{tikzcd}
\widehat{R_U} \arrow{r} \arrow[swap]{d} & \widehat{R_\wp} \arrow{d} \\
\widehat{R_U}/\eta \arrow{r} & k(\wp).
	\end{tikzcd}
\] 
Since $\widehat{R_\wp}$ is Henselian, the category of finite étale $\widehat{R_\wp}$-algebras is equivalent to the category of finite étale $k(\wp)$-algebras via the functor $B\mapsto \overline{B}:=B\otimes_{\widehat{R_\wp}}k(\wp)$ (see Proposition 1.4.4 in \parencite{milne2016etale}). Since $G$-Galois $R$-algebras are also finite étale $R$-algebras, $(B_U\otimes_{\widehat{R_U}}\widehat{R_\wp})\otimes_{\widehat{R_\wp}}k(\wp)\cong \overline{B_\wp}$ implies that $B_U\otimes_{\widehat{R_U}}\widehat{R_\wp}\cong B_\wp$ as $G$-Galois $\widehat{R_\wp}$-algebras. Therefore, if we let $E_U:=B_U\otimes_{\widehat{R_U}}F_U$, then we see that 
\[E_U\otimes_{F_U}F_\wp=(B_U\otimes_{\widehat{R_U}}F_U)\otimes_{F_U}F_\wp\cong B_\wp\otimes_{\widehat{R_\wp}}F_\wp\cong E\otimes_F F_\wp.\] 

Now we can build a matrix algebra over $F_U$ containing $E_U$ by using a crossed-product construction. As described in \parencite{auslander1960brauer} (in the discussion immediately following the proof of Theorem A.9) we can form the crossed-product algebra $A_U:=\Delta(E_U, G, 1)$ (in \parencite{auslander1960brauer} denoted $\Delta(1; E_U; G)$) which can be thought of as $\bigoplus_{\sigma\in G} E_Uu_\sigma$ where multiplication is defined by $(au_\sigma)(bu_\tau)=a\sigma(b)u_{\sigma\tau}$. The class $[\Delta(E_U, G, 1)]$ is trivial in the Brauer group of $F_U$ so as an $F_U$-algebra it is isomorphic to $M_{p^m}(F_U)$ and we can view $E_U$ as a $G$-Galois $F_U$-subalgebra of $A_U$ by the isomorphism $\iota:E_U\rightarrow E_Uu_{id}$. 

By assumption (iii), $D\otimes_{F_P}F_\wp\cong M_{p^m}(F_\wp)\cong A_U\otimes_{F_U}F_\wp$. Since we also have an isomorphism $\phi:E_U\otimes_{F_U}F_\wp\rightarrow E\otimes_F F_\wp$ where $E_U\otimes_{F_U}F_\wp$ is a $G$-Galois $F_\wp$-subalgebra of $A_U\otimes_{F_U}F_\wp$ and $E\otimes_F F_\wp$ is a $G$-Galois $F_\wp$-subalgebra of $D\otimes_{F_P}F_\wp$ with $\vert G\vert=\text{deg}(A_U\otimes_{F_U}F_\wp)=p^m$, by Lemma \ref{restriction} we know there exists an isomorphism $\psi:A_U\otimes_{F_U}F_\wp\rightarrow D\otimes_{F_P}F_\wp$ such that the following diagram commutes:
\[
	\begin{tikzcd}
E_U\otimes_{F_U}F_\wp \arrow{r}{\phi} \arrow[swap]{d}{\iota\otimes id} & E\otimes_{F_P}F_\wp \arrow{d} \\
A_U\otimes_{F_U}F_\wp \arrow{r}{\psi} & D\otimes_{F_P}F_\wp
	\end{tikzcd}
.\] 

\end{proof}

\begin{lemma}\label{compatibility2}
Suppose $A_U$ is a split central simple algebra over $F_U$ containing a $G$-Galois $F_U$-subalgebra $E_U$ where $\vert G\vert=p^m$. Then there exists a central simple algebra $A_{P'}$ over $F_{P'}$ with a $G$-Galois $F_{P'}$-subalgebra $E_{P'}$ such that:
\begin{enumerate}
	\item[i.] $A_{P'}\otimes_{F_{P'}}F_{\wp_U}\cong A_U\otimes_{F_U}F_{\wp_U}$ and the isomorphism restricts to an isomorphism of $G$-Galois $F_{\wp_U}$-algebras, $E_{P'}\otimes_{F_{P'}}F_{\wp_U}\cong E_U\otimes_{F_U}F_{\wp_U}$
	\item[ii.] $A_{P'}\otimes_{F_{P'}}F_{\wp_{U'}}\cong M_{p^m}(F_{\wp_{U'}})$ and the isomorphism restricts to an isomorphism of $G$-Galois $F_{\wp_{U'}}$-algebras, $E_{P'}\otimes_{F_{P'}}F_{\wp_{U'}}\cong F_{\wp_{U'}}^{p^m}$.
\end{enumerate}
\end{lemma}

\begin{proof}
As in the previous proof, we construct $E_{P'}$ first. Note that since $E_{\wp_U}:=E_U\otimes_{F_U}F_{\wp_U}$ is a $G$-Galois $F_{\wp_U}$-algebra and $\text{char}(F_{\wp_U})=p$, the existence of a generic $G$-Galois extension $S(G)/R(G)$ (Theorem \ref{genericgalois}) again implies that there exists a ring homomorphism $f_{E_{\wp_U}}:R(G)\rightarrow F_{\wp_U}$ such that $E_{\wp_U}\cong S(G)\otimes_{f_{E_{\wp_U}}}F_{\wp_U}$. Recall that $R(G)=\mathbb{F}_p[x_1,x_2,\dots,x_m]$. Similarly, since $F_{\wp_{U'}}^{p^m}$ is a $G$-Galois $F_{\wp_{U'}}$-algebra, there exists a ring homomorphism $f_{E_{\wp_{U'}}}:R(G)\rightarrow F_{\wp_{U'}}$ such that $F_{\wp_{U'}}^{p^m}\cong S(G)\otimes_{f_{E_{\wp_{U'}}}}F_{\wp_{U'}}$ as $G$-Galois $F_{\wp_{U'}}$-algebras.

Now we note that $F_{\wp_U}$ is the completion of $F_{P'}$ with respect to a height one discrete valuation (with valuation ring $(\widehat{R_{P'}})_{\wp_U}$) and $F_{\wp_{U'}}$ is the completion of $F_{P'}$ with respect to a different height one discrete valuation (with valuation ring $(\widehat{R_{P'}})_{\wp_{U'}}$). Since these are distinct height one valuations on $F_{P'}$ they are independent and so $F_{P'}$ is simultaneously dense in $F_{\wp_U}$ and $F_{\wp_{U'}}$ (see Theorem VI.7.2.1 of \parencite{bourbaki1972commutative}). 

Lemma \ref{krasners} implies that there exists a $\delta_1>0$ such that if $b_i\in F_{\wp_U}$ satisfy $\norm{b_i-f_{E_{\wp_U}}(x_i)}_{\wp_U}<\delta_1$, then the ring homomorphism $f':R(G)\rightarrow F_{\wp_U}$ defined by setting $f'(x_i)=b_i$ has the property that $S(G)\otimes_{f_{E_{\wp_U}}}F_{\wp_U}\cong S(G)\otimes_{f'} F_{\wp_U}$ as $G$-Galois $F_{\wp_U}$-algebras. Similarly, there exists a $\delta_2>0$ such that if $c_i\in F_{\wp_{U'}}$ satisfy $\norm{c_i-f_{E_{\wp_{U'}}}(x_i)}_{\wp_{U'}}<\delta_2$, then the ring homomorphism $f'':R(G)\rightarrow F_{\wp_{U'}}$ defined by setting $f''(x_i)=c_i$ has the property that $S(G)\otimes_{f_{E_{\wp_{U'}}}}F_{\wp_{U'}}\cong S(G)\otimes_{f''} F_{\wp_{U'}}$ as $G$-Galois extensions of $F_{\wp_{U'}}$.

Since $F_{P'}$ is simultaneously dense in $F_{\wp_U}$ and $F_{\wp_{U'}}$, letting $\delta=\text{min}\{\delta_1,\delta_2\}$, we know that for all $1\leq i\leq m$, there exists $d_i\in F_{P'}$ such that $\norm{d_i-f_{E_{\wp_U}}(x_i)}_{\wp_U}<\delta_1$ and $\norm{d_i-f_{E_{\wp_{U'}}}(x_i)}_{\wp_{U'}}<\delta_2$. We define a ring homomorphism $f_{P'}:R(G)\rightarrow F_{P'}$ by setting $f_{P'}(x_i)=d_i$ and this homomorphism can also be thought of as a ring homomorphism from $R(G)$ to $F_{\wp_U}$ or $F_{\wp_{U'}}$.

By our choice of $\delta$ we have $G$-equivariant isomorphisms, 
\[S(G)\otimes_{f_{P'}}F_{\wp_U}\cong S(G)\otimes_{f_{E_{\wp_U}}}F_{\wp_U},\text{ and }S(G)\otimes_{f_{P'}}F_{\wp_{U'}}\cong S(G)\otimes_{f_{E_{\wp_{U'}}}}F_{\wp_{U'}}.\] 

Setting $E_{P'}:=S(G)\otimes_{f_{P'}}F_{P'}$ this tells us that when we extend scalars up to each of the branch fields we get $G$-equivariant isomorphisms $E_{P'}\otimes_{F_{P'}}F_{\wp_U}\cong E_U\otimes_{F_U}F_{\wp_U}$ and $E_{P'}\otimes_{F_{P'}}F_{\wp_U}\cong F_{\wp_{U'}}^{p^m}$. 

Now we set $A_{P'}:=\Delta(E_{P'}, G, 1)=\bigoplus_{\sigma\in G} E_{P'}u_\sigma$, where multiplication is again defined by $(au_\sigma)(bu_\tau)=a\sigma(b)u_{\sigma\tau}$. Since $A_{P'}$ and $A_U$ are both split, we have that $A_{P'}\otimes_{F_{P'}}F_{\wp_U}\cong A_U\otimes_{F_U}F_{\wp_U}$ and $A_{P'}\otimes_{F_{P'}}F_{\wp_{U'}}\cong M_{p^m}(F_{\wp_{U'}})$. Identifying $E_{P'}$ with $E_{P'}u_{id}$ shows that $E_{P'}$ is a $G$-Galois $F_{P'}$-subalgebra of $A_{P'}$. We also have $G$-equivariant isomorphisms $\phi_U:E_U\otimes_{F_U}F_{\wp_U}\rightarrow E_{P'}\otimes_{F_{P'}}F_{\wp_U}$, and $\phi_{U'}:F_{\wp_{U'}}^{p^m}\rightarrow E_{P'}\otimes_{F_{P'}}F_{\wp_{U'}}$ as explained above. Therefore, as in the proof of Lemma \ref{compatibility1}, we can use Lemma \ref{restriction} to conclude that there exist isomorphisms $\psi_U$ and $\psi_{U'}$  that restrict to the isomorphisms $\phi_U$ and $\phi_{U'}$ respectively on the given $G$-Galois subalgebras, i.e. such that the following two diagrams both commute: 

\[
	\begin{tikzcd}
E_U\otimes_{F_U}F_{\wp_U} \arrow{r}{\phi_U} \arrow[swap]{d}{\iota\otimes id} & E_{P'}\otimes_{F_{P'}}F_{\wp_U} \arrow{d} & \text{and} & F_{\wp_{U'}}^{p^m} \arrow{r}{\phi_{U'}} \arrow[swap]{d}{diag} & E_{P'}\otimes_{F_{P'}}F_{\wp_{U'}} \arrow{d} \\
A_U\otimes_{F_U}F_{\wp_U} \arrow{r}{\psi_U} & A_{P'}\otimes_{F_{P'}}F_{\wp_U} & & M_{p^m}(F_{\wp_{U'}}) \arrow{r}{\psi_{U'}} & A_{P'}\otimes_{F_{P'}}F_{\wp_{U'}}
	\end{tikzcd}
.\] 

\end{proof}

\section{Main Theorem}\label{mainsection}

We will now prove the main theorem using a patching argument which follows the strategy employed by Harbater, Hartmann and Krashen in section 4 of \parencite{harbater2011patching}.

In the proof we will use induced algebras so we record the definition of an induced algebra in this context below for convenience.

\begin{definition}\label{inducedalg}
If $F$ is a field, $G$ is a finite group with subgroup $H$, and $E$ is an $H$-Galois $F$-algebra, then we define the induced $G$-Galois $F$-algebra, $\text{Ind}_H^G E$, as follows. Let $(G:H)=m$ and fix $C=\{\sigma_1, \sigma_2,\dots, \sigma_m\}$, a set of left coset representatives. As an $F$-algebra, $\text{Ind}_H^G E$ is the direct product of $m$ copies of $E$ indexed by the elements of $C$. To define the $G$-action we note that given $\sigma\in G$, and $\sigma_i\in C$, we have $\sigma\sigma_i\in\sigma_{j(i)}H$ for some unique coset $\sigma_{j(i)}H$ depending on $i$. So there is some $h_i\in H$ such that $\sigma\sigma_i=\sigma_{j(i)}h_i$ and we define \[\sigma\cdot (c_{\sigma_1}, c_{\sigma_2},\dots,c_{\sigma_m}):=(x_{\sigma_1}, x_{\sigma_2}, \dots, x_{\sigma_m}),\] where $x_{\sigma_{j(i)}}=h_i\cdot c_{\sigma_i}$. This action permutes the entries by coset and acts within each copy of $E$ by the original $H$-action.
\end{definition}

\begin{theorem}\label{mainresult}
Let $K$ be a complete discretely valued field with algebraically closed residue field $k$, and let F be a finitely generated field extension of K with transcendence degree 1. Assume further that $\text{char}(K)=\text{char}(k)$. A finite group $G$ is admissible over $F$ if and only if the $q$-Sylow subgroups of $G$ are abelian of rank at most 2 for every prime $q\neq \text{char}(k)$. 
\end{theorem}
\begin{proof}
In the case that $\text{char}(k)\nmid \vert G\vert$, this result is precisely Theorem 4.5 in \parencite{harbater2011patching}, and in the case $\text{char}(k)=p\mid \vert G\vert$ the forward direction is implied by Proposition 3.5 in \parencite{harbater2011patching}. Therefore we only need to prove that the conditions placed on the $q$-Sylow subgroups of $G$ are sufficient in the $\text{char}(k)=p\mid \vert G\vert$ case. 

First we set some notation. Let $p_1, p_2, \dots, p_r$ be all of the distinct primes dividing $n:=\vert G\vert$ and order them so that $p_1=p$. So we can write $n=p_1^{\alpha_1}p_2^{\alpha_2}\cdots p_l^{\alpha_l}$, where $\alpha_i>0$. For each prime $p_i$, we fix a $p_i$-Sylow subgroup $H_i$ of $G$.

We begin with an outline of the proof that follows. First we obtain a good model of $F$ and check that it satisfies the conditions imposed in Section \ref{compatibilitysection}. Then we build an adequate extension for each Sylow subgroup over a point in the model, for the $p$-Sylow subgroup this will be done over a special point in the model. We will then invoke results from Section \ref{compatibilitysection} to build compatible structures on a portion of the model. In the next step we explain how to fill in the rest of the model with trivial patches for the maximal separable subalgebras and show that the induced structures over all of the fields in our framework are compatible and therefore form a patching problem. We patch these together to obtain a $G$-Galois $F$-algebra which we call $L$. This process is repeated with the central simple algebras that contain the maximal separable algebras. Finally, the patched central simple algebra over $F$, which we call $A$, is shown to be a division algebra which then implies that $L$ is a maximal subfield of $A$, completing the proof.\\

\noindent\textit{Establishing a Good Model}\\

Let $\widehat{X}$ be a regular model of $F$ and let $X$ denote the closed fiber of $\widehat{X}$. For each $1\leq i\leq l$ we pick a distinct closed point on $X$, $P_i$, at which $X$ is regular. Now let $b:\widehat{X}_B\rightarrow \widehat{X}$ be the blow-up of $\widehat{X}$ at $P_1$. We note that $\widehat{X}_B$ is also a regular model of $F$.

We let $X_B$ denote the closed fiber of $\widehat{X}_B$ and note that $X_B$ is the union of $b^{-1}(X\backslash\{P_1\})$ and $b^{-1}(P_1)$ where $b^{-1}(X\backslash\{P_1\})\cong X\backslash\{P_1\}$ and $b^{-1}(P_1)\cong \mathbb{P}_{k(P_1)}^1$. We also know that the strict transform of $X$ in $\widehat{X}_B$, $W:=\overline{b^{-1}(X\backslash\{P_1\})}$, is isomorphic to $X$. This is because we can apply Proposition IV-21 in \parencite{eisenbud2006geometry} to the closed subscheme $X$ of $\widehat{X}$. The proposition implies that $b\vert_W:W\rightarrow X$ is the blow-up of $X$ at $P_1$. Since $P_1$ is already an effective Cartier divisor of $X$, the universal property of the blow-up implies that $b\vert_W$ is an isomorphism.
 
So, for each $1\leq i\leq l$, we let $P'_i$ denote the unique point in $W$ that is mapped to $P_i$ by $b$. Each $P'_i$ is a distinct closed point on $X_B$, and $X_B$ is regular at $P'_i$ for $2\leq i\leq l$. Note that $P'_1$ is a point on $b^{-1}(P_1)$ and we can pick another closed point, $P$ on $b^{-1}(P_1)$ different from $P'_1$. 

We let $\mathscr{P}$ be the set consisting of all points on $X_B$ at which irreducible components of $X_B$ meet, as well as the points $P'_1, P'_2, \dots, P'_l$ and $P$. Then we have the set $\mathscr{W}$ of all connected components of $X_B\backslash \mathscr{P}$, and the set $\mathscr{B}$ of all branches of $X_B$ at the points in $\mathscr{P}$ as described in Notation \ref{modelsystem}. At each $\xi\in I:=\mathscr{P}\cup\mathscr{W}\cup\mathscr{B}$ we also have the associated field $F_\xi$, and ring $\widehat{R_\xi}$, also defined in Notation \ref{modelsystem}.\\

\noindent\textit{Building Sylow Adequate Extensions}\\

For $2\leq i\leq l$, $P'_i$ is a closed point at which $X_B$ is regular. Therefore, there is a unique branch $\wp_i$ of $X_B$ at $P'_i$. We also know that $F$ (and therefore $F_{P'_i}$) contains a primitive $p_i^{\alpha_i}$-th root of unity. This is because $p_i\neq \text{char}(k)$ and $k$ is algebraically closed. By assumption, $H_i$, a $p_i$-Sylow subgroup of $G$, is an abelian group of rank at most 2 and so there exists a field extension $E_{P'_i}/F_{P'_i}$ and a division algebra $D_{P'_i}$ central over $F_{P'_i}$ containing $E_{P'_i}$ as a maximal subfield such that $E_{P'_i}/F_{P'_i}$ is $H_i$-Galois and $E_{P'_i}\otimes_{F_{P'_i}}F_{\wp_i}$ is split. The constructions of these extensions and division algebras is described in detail in the proof of proposition 4.4 in \parencite{harbater2011patching}.

Now we note that $X_B$ is also regular at $P$ and so there is a unique branch of $X_B$ at $P$ which we call $\wp$. We know that $\widehat{R_P}$ is an equicharacteristic complete regular local ring of dimension 2 with residue field $k(P)$. Since $H_1$ is a $p$-Sylow subgroup of $G$ it is a finite p-group, so by Proposition \ref{specialcase} there exists an $H_1$-Galois field extension $E_P/F_P$ and a division algebra $D_P$ that contains $E_P$ as a maximal subfield such that $E_P$ is unramified at $\wp$ and $D_P$ splits over $\wp$.\\

\noindent\textit{Compatibility on the Exceptional Fiber}\\

Since $b^{-1}(P_1)\cong \mathbb{P}_{k(P_1)}^1$, and $\mathscr{P}\cap b^{-1}(P_1)=\{P,P'_1\}$, we know that $b^{-1}(P_1)\backslash \{P,P'_1\}$ (which is isomorphic to the affine line over $k(P_1)$ with one point removed) is a connected component of $X_B\backslash \mathscr{P}$. In other words, $U:=b^{-1}(P_1)\backslash \{P,P'_1\}$ is an element of $\mathscr{W}$. Since $P'_1$ is in the closure of $U$, there is a branch of $X_B$ at $P'_1$ lying along $U$, which we denote $\wp_U$. 

We observe that $P'_1$ is contained in the closure of exactly one other connected component of $X_B\backslash \mathscr{P}$. This is because every element of $\mathscr{W}$ other than $U$ is contained in $W$ and is therefore isomorphic to a connected component of $X\backslash\{b(Q):Q\in\mathscr{P}\cap W\}$. So if $U'$ is an element of $\mathscr{W}$ with $P'_1\in\overline{U'}$ and $U'\neq U$, then $b(U')$ is a connected component of $X\backslash\{b(Q):Q\in\mathscr{P}\cap W\}$ with $P_1\in\overline{b(U)}$, with the closure taken in $X$. Since $X$ is regular at $P_1$, there is only one such component. We let $U'$ denote the unique element of $\mathscr{W}$ described above, and we let $\wp_1$ denote the branch of $X_B$ at $P'_1$ along $U'$. This discussion illustrates that the semi-global patching framework $(\widehat{X}_B, \mathscr{I},\mathscr{F})$ satisfies the conditions outlined in the beginning of Section \ref{compatibilitysection}. The notation we use is only slightly different from that used in Section \ref{compatibilitysection}, with $P'_1$ corresponding to $P'$ and $\wp_1$ corresponding to $\wp_{U'}$ in that section.

By the lemmas in Section \ref{compatibilitysection}, we have $H_1$-Galois algebras $E_U$ and $E_{P'_1}$ over $F_U$ and $F_{P'_1}$ respectively. These are embedded inside algebras $C_U:=\Delta(E_U, H_1, 1)$ and $C_{P'_1}:=\Delta(E_{P'_1}, H_1, 1)$ which are central simple algebras over the respective fields. Furthermore, these central simple algebras and the designated subalgebras satisfy the compatibility conditions over $F_\wp$ that were outlined in those lemmas. We let $E_\wp:=E_P\otimes_{F_P}F_\wp$ and $E_{\wp_U}:=E_{P'_1}\otimes_{F_{P'_1}}F_{\wp_U}$, and we note that $E_{P'_1}$ splits over $\wp_1$ by construction, so $E_{P'_1}\otimes_{F_{P'_1}}F_{\wp_1}\cong F_{\wp_1}^{\vert H_1\vert}$.\\

\noindent\textit{Filling in Trivial Patches}\\

For all of the remaining fields in the patching framework (i.e. for all $F_\xi$ with $\xi\in I\backslash\{P'_1, P'_2, \dots ,P'_l, P, U, \wp,\wp_U\}$) we set $E_\xi:=F_\xi$. For consistency, we will use $H_\xi$ to denote the subgroup of $G$ associated to the $F_\xi$-algebra $E_\xi$ for all $\xi\in I$. So $H_{P'_i}=H_i$ for all $1\leq i\leq l$, additionally $H_P=H_U=H_\wp=H_{\wp_U}=H_1$ and for all other $\xi\in I$, $H_\xi=1$. Now we define for all $\xi\in I$ a $G$-Galois $F_\xi$-algebra by setting $L_\xi:=\text{Ind}_{H_\xi}^G E_\xi$ as in Definition \ref{inducedalg}. 

Recall that given $\xi\in \mathscr{B}$, $\xi$ is a branch of $X_B$ at a point $Q\in\mathscr{P}$ over a connected component $V\in\mathscr{W}$, so we can associate to $\xi$ a unique element $(Q,V)\in\mathscr{P}\times\mathscr{W}$. To patch the $G$-Galois $F_\xi$-algebras $L_\xi$ defined above, and obtain a $G$-Galois $F$-algebra, we just need to check that they are compatible in the following sense. For each $\xi\in\mathscr{B}$, we show that there exist $G$-equivariant $F_\xi$-algebra isomorphisms $\phi_{Q,\xi}:L_Q\otimes_{F_Q}F_\xi\rightarrow L_\xi$ and $\phi_{V,\xi}:L_V\otimes_{F_V}F_\xi\rightarrow L_\xi$, where $(Q,V)$ is the element of $\mathscr{P}\times\mathscr{W}$ associated to $\xi$.

In the case that $\xi\in\mathscr{B}\backslash\{\wp,\wp_U\}$, we know that the associated pair $(Q,V)\in\mathscr{P}\times\mathscr{W}$ is such that $Q\neq P$ and $V\neq U$ so we have an $H_Q$-equivariant isomorphism $\phi_{Q,\xi}^E: E_Q\otimes_{F_Q}F_\xi\rightarrow F_\xi^{\vert H_Q\vert}=\text{Ind}_{1}^{H_Q}F_\xi$ which induces a $G$-equivariant isomorphism $\tilde{\phi}_{Q,\xi}^E:\text{Ind}_{H_Q}^G(E_Q\otimes_{F_Q}F_\xi)\rightarrow \text{Ind}_{H_Q}^G(\text{Ind}_1^{H_Q}F_\xi)$. When we compose these maps with the isomorphisms
\[\text{Ind}_{H_Q}^G(E_Q\otimes_{F_Q}F_\xi)\cong (\text{Ind}_{H_Q}^G E_Q)\otimes_{F_Q}F_\xi\text{ and }\text{Ind}_{H_Q}^G(\text{Ind}_1^{H_Q}F_\xi)\cong \text{Ind}_1^G F_\xi\cong F_\xi^{\vert G\vert},\] 
we get an isomorphism $\phi_{Q,\xi}: L_Q\otimes_{F_Q}F_\xi\rightarrow L_\xi$. Similarly, $\phi_{V,\xi}^E:E_V\otimes_{F_V}F_\xi\xrightarrow{\sim} F_\xi$ so we have an induced $G$-equivariant isomorphism $\tilde{\phi}_{V,\xi}^E:\text{Ind}_1^G (E_V\otimes_{F_V}F_\xi) \rightarrow \text{Ind}_1^G F_\xi$ which gives a $G$-equivariant isomorphism $\phi_{V,\xi}:L_V\otimes_{F_V}F_\xi\rightarrow L_\xi$.

In the case that $\xi=\wp$, which is associated to $(P,U)$, we have that $E_P\otimes_{F_P}F_\wp=E_\wp$ by definition. We take $\phi_{P,\wp}$ to be the map induced by the identity map since, $L_\wp=\text{Ind}_{H_1}^G E_\wp$. Additionally, the compatibility established in Lemma \ref{compatibility1} implies that $E_U\otimes_{F_U}F_\wp\cong E_\wp$. This isomorphism of $F_\wp$-algebras is $H_1$-equivariant so the induced isomorphism $\phi_{U,\wp}:L_U\otimes_{F_U}F_\wp\rightarrow L_\wp$ is $G$-equivariant. 

The same argument holds for $\xi=\wp_U$, associated to $(P'_1,U)$, where the $H_1$-equivariant isomorphism $E_U\otimes_{F_U}F_{\wp_U}\cong E_{P'_1}\otimes_{F_{P'_1}}F_{\wp_U}=E_{\wp_U}$ is established in Lemma \ref{compatibility2}. Therefore by Lemma 4.2 in \parencite{harbater2011patching}, there exists a $G$-Galois $F$-algebra, $L$, with the property that for all $\xi\in I$, $L\otimes_F F_\xi\cong L_\xi$ (as $G$-Galois $F_\xi$-algebras) in a manner compatible with the maps $\phi_{Q,\xi}$ and $\phi_{V,\xi}$ for each $\xi\in\mathscr{B}$ associated to $(Q,V)$.\\

\noindent\textit{Repeated Process with CSA's}\\

Now for each $\xi\in I$ we construct a central simple algebra over $F_\xi$. For $\xi\in\{P'_2, P'_3, \dots, P'_l, P\}$ we set $A_\xi:=\text{Mat}_{(G:H_\xi)}(D_\xi)$. Since $E_\xi$ is a maximal commutative subalgebra of $D_\xi$, the subalgebra of $A_\xi$ consisting of all diagonal matrices with entries in $E_\xi$ is a maximal commutative subalgebra of $A_\xi$. Furthermore, this subalgebra is isomorphic to $E_\xi^{(G:H_\xi)}$, which we give the structure of a $G$-Galois $F_\xi$-algebra by defining the $G$-action on $E_\xi^{(G:H_\xi)}$ to be the $G$-action on $\text{Ind}_{H_\xi}^G E_\xi=L_\xi$. Therefore we have a $G$-equivariant injective homomorphism $i_\xi:L_\xi\rightarrow A_\xi$ such that the image of $i_\xi$ is the subalgebra of $A_\xi$ of diagonal matrices with entries in $E_\xi$.  

Similarly, for $\xi\in\{P'_1, U\}$, we set $A_\xi:=\text{Mat}_{(G:H_\xi)}(C_\xi)$ and again we have an embedding of $L_\xi$ in $A_\xi$ as the set of diagonal matrices with entries in $E_\xi$, a maximal commutative subalgebra. For $\xi=\wp$ we set $A_\wp:=\text{Mat}_{(G:H_\wp)}(D_P\otimes_{F_P}F_\wp)$ where again $L_\wp=\text{Ind}_{H_\wp}^G E_\wp$ can be embedded in $A_\wp$ as a maximal commutative subalgebra by mapping to diagonal matrices with entries in $E_\wp=E_P\otimes_{F_P}F_\wp$. Similarly, we set $A_{\wp_U}:=\text{Mat}_{(G:H_{\wp_U})}(C_{P'_1}\otimes_{F_{P'_1}}F_{\wp_U})$ and $L_{\wp_U}$ is embedded in $A_{\wp_U}$ as diagonal matrices with entries in $E_{\wp_U}$. 

For all $\xi\in I\backslash\{P'_1, P'_2, \dots ,P'_l, P, U, \wp,\wp_U\}$ we set $A_\xi:=\text{Mat}_n (F_\xi)$ and we can embed $L_\xi=F_\xi^{n}$ in $A_\xi$ as the subalgebra of diagonal matrices. In all of the cases above we have defined an embedding, $\iota_\xi:L_\xi\hookrightarrow A_\xi$, as a maximal commutative subalgebra of $A_\xi$. 

We now need to show that for every $\xi=(Q,V)\in\mathscr{B}$, we have compatibility isomorphisms $\psi_{Q,\xi}:A_Q\otimes_{F_Q}F_\xi\rightarrow A_\xi$ and $\psi_{V,\xi}:A_V\otimes_{F_V}F_\xi\rightarrow A_\xi$ such that the following diagrams commute:

\[
	\begin{tikzcd}
L_Q\otimes_{F_Q}F_\xi \arrow{r}{\phi_{Q,\xi}} \arrow[swap]{d}{\iota_Q\otimes id} & L_\xi \arrow{d}{\iota_\xi} & \text{ and } & L_V\otimes_{F_V}F_\xi \arrow{r}{\phi_{V,\xi}} \arrow[swap]{d}{\iota_V\otimes id} & L_\xi \arrow{d}{\iota_\xi} \\
A_Q\otimes_{F_Q}F_\xi \arrow{r}{\psi_{Q,\xi}} & A_\xi & & A_V\otimes_{F_V}F_\xi \arrow{r}{\psi_{V,\xi}} & A_\xi
	\end{tikzcd}
\] 

Following the argument made in Lemma 4.2 of \parencite{harbater2011patching} we first establish that isomorphisms of $F_\xi$-algebras, $\tilde{\psi}_{Q,\xi}:A_Q\otimes_{F_Q}F_\xi\rightarrow A_\xi$ and $\tilde{\psi}_{V,\xi}:A_V\otimes_{F_V}F_\xi\rightarrow A_\xi$, exist for each $\xi=(Q,V)\in\mathscr{B}$. The existence of $\tilde{\psi}_{Q,\xi}$ follows from the fact that $A_Q$ is defined to be a matrix algebra with entries in a central simple algebra over $F_Q$ which splits over $F_\xi$. Therefore $A_Q\otimes_{F_Q}F_\xi$, which is isomorphic to a matrix algebra with entries in a matrix algebra over $F_\xi$, is a split central simple algebra over $F_\xi$ of degree $n$. So $A_Q\otimes_{F_Q}F_\xi$ is isomorphic to $\text{Mat}_n(F_\xi)$. We now observe that $A_\xi$ is also a split central simple algebra of degree $n$ for every $\xi\in \mathscr{B}$. Recall that $A_\xi$ is constructed as a matrix algebra over a split central simple algebra in such a way that the degree is $n$. Since we have isomorphisms, 
\[A_Q\otimes_{F_Q}F_\xi\cong \text{Mat}_n(F_\xi)\cong A_\xi,\]
there exists an isomorphism $\tilde{\psi}_{Q,\xi}:A_Q\otimes_{F_Q}F_\xi\rightarrow A_\xi$. Similarly, $A_V$ is also a degree $n$ split central simple algebra over $F_V$ for each $V\in\mathscr{W}$, and so we also have an isomorphism $\tilde{\psi}_{V,\xi}:A_V\otimes_{F_V}F_\xi\rightarrow A_\xi$.

Having established the existence of the isomorphisms $\tilde{\psi}_{Q,\xi}$, and $\tilde{\psi}_{V,\xi}$, for every $\xi=(Q,V)\in I$, we can now follow the rest of the argument made in Lemma 4.2 of \parencite{harbater2011patching}. We observe that since $L_\xi$ is a $G$-Galois $F_\xi$-algebra, it is a Frobenius algebra. Via $\iota_\xi$ we can view $L_\xi$ as a Frobenius subalgebra of $A_\xi$ of dimension $n$, equal to the degree of $A_\xi$. Now we consider the isomorphism $\tilde{\psi}_{Q,\xi}\circ(\iota_Q\otimes id)\circ\phi^{-1}_{Q,\xi}$ of $L_\xi$ into $A_\xi$. By Theorems 2.2.2 and 2.2.3 in \parencite{jacobson2009finite}, this isomorphism can be extended to an inner automorphism of $A_\xi$, which we call $\alpha_{Q,\xi}$. If we set $\psi_{Q,\xi}:=\alpha^{-1}_{Q,\xi}\circ\tilde{\psi}_{Q,\xi}$ we see that this is a an isomorphism $\psi_{Q,\xi}:A_Q\otimes_{F_Q}F_\xi\rightarrow A_\xi$ such that 

\[
	\begin{tikzcd}
L_Q\otimes_{F_Q}F_\xi \arrow{r}{\phi_{Q,\xi}} \arrow[swap]{d}{\iota_Q\otimes id} & L_\xi \arrow{d}{\iota_\xi} \\
A_Q\otimes_{F_Q}F_\xi \arrow{r}{\psi_{Q,\xi}} & A_\xi
	\end{tikzcd}
\] 

Since the proof of the existence of $\tilde{\psi}_{V,\xi}$ is identical, we now conclude by patching that there exists a central simple $F$-algebra, $A$, containing $L$ as a maximal commutative subalgebra, with the property that $A\otimes_F F_\xi\cong A_\xi$ for every $\xi\in I$ in a manner compatible with the maps $\psi_{Q,\xi}$ and $\psi_{V,\xi}$ for every $\xi=(Q,V)\in\mathscr{B}$.\\

\noindent\textit{Patched Algebra is Division}\\

We observe that $A$ is in fact a division algebra over $F$. This is because the degree of $A$ is equal to the index of $A$. We see this by noting that for $\xi\in\{P'_2, P'_3, \dots, P'_l, P\}$, we have $\text{ind}(A_\xi)=n/(G:H_\xi)$. This is because over these points $A_\xi$ is constructed as a matrix algebra over a division algebra $D_\xi$ and the degree of $D_\xi$ is equal to the degree of its maximal subfield $E_\xi$ (as an extension of $F_\xi$) which is $n/(G:H_\xi)$. In the set $\{H_\xi:\xi\in\{P'_2, P'_3, \dots, P'_l, P\}\}$ we have a Sylow subgroup of $G$ for every prime divisor of $n=\vert G\vert$ so $n=\text{lcm}\{n/(G:H_\xi):\xi\in\{P'_2, P'_3, \dots, P'_l, P\}\}$. We also know that $\text{ind}(A_\xi)\mid\text{ind}(A)$ (Proposition 13.4 (iv) in \parencite{pierce1982associative}) and that $\text{ind}(A)\mid\text{deg}(A)$. Taken all together this shows that \[\text{deg}(A)=n=\text{lcm}\{\text{ind}(A_\xi):\xi\in\{P'_2, P'_3, \dots, P'_l, P\}\}\mid\text{ind}(A)\mid\text{deg}(A).\] 

Therefore it must be that $\text{deg}(A)=\text{ind}(A)$ so $A$ is a division algebra over $F$. This in turn implies that $L$, as a commutative subalgebra of a finite dimensional division algebra, is actually a domain and therefore a field. Since $L$ is a maximal commutative subalgebra it must be a maximal subfield of $A$, and as it is $G$-Galois over $F$, we conclude that $G$ is $F$-admissible.

\end{proof}

\printbibliography
\end{document}